\documentclass{amsart}
\usepackage{amssymb,latexsym}
\usepackage{amsmath}
\usepackage[dvipdfm]{graphicx}

\def\cro{{\mbox {\sc cr}}}

\theoremstyle{plain}
\newtheorem{thm}{\bfseries Theorem}
\newtheorem{lemma}[thm]{\bfseries Lemma}        
\newtheorem{cor}[thm]{\bfseries Corollary}

\newtheorem{conj}[thm]{\bfseries Conjecture}

\begin{document}

\title[]{Towards the Albertson conjecture}

\author[J. Bar\'at]{J\'anos Bar\'at}
\address{Department of Computer Science and Systems Technology, University of Pannonia, Egyetem u. 10, 8200 Veszpr\'em, Hungary}
\thanks{Research is supported by OTKA Grant PD~75837.}

\author[G. T\'oth]{G\'eza T\'oth}
\address{R\'enyi Institute, Re\'altanoda u. 13-15, 1052 Budapest, Hungary}
\thanks{Research is supported by OTKA T 038397 and T 046246}




\date{\today}

\begin{abstract}
Albertson conjectured that if a graph $G$ has chromatic number $r$ then its
crossing number is at least as much as the crossing number of $K_r$.
Albertson, Cranston, and Fox verified the conjecture for $r\le 12$.
We prove the statement for $r\le 16$.
\end{abstract}

\maketitle

\centerline{\em Dedicated to the memory of Michael O. Albertson.}

\section{Introduction}
Graphs in this paper are without loops and multiple edges.
Every planar graph is four-colorable by the Four Color Theorem \cite{ah, rsst}.
Efforts to solve the Four Color Problem had a great effect on the development
of graph theory, and it is one of the most important theorems of the field.

The {\em crossing number} $\cro(G)$ of a graph $G$ is the minimum number of 
edge crossings in a drawing of $G$ in the plane.
It is a natural relaxation of planarity, see \cite{sz} for a survey. 
The {\em chromatic number} $\chi(G)$ of a graph $G$ is the minimum number of 
colors in a proper coloring of $G$.
The Four Color Theorem states if $\cro(G)=0$ then $\chi(G)\le 4$.
Oporowski and Zhao \cite{oz} proved that every graph with crossing number at most two is
5-colorable.
Albertson et al. \cite{alb2} showed that if $\cro(G)\le 6$, then $\chi(G)\le 6$.
It was observed by Schaefer that if $\cro(G)=k$ then $\chi(G)=O(\sqrt[4]{k})$
and this bound cannot be improved asymptotically \cite{alb}.

It is well-known that graphs with chromatic number $r$ do not necessarily
contain $K_r$ as a subgraph, they can have clique number 2 
\cite{z}. 
The Haj\'os conjecture proposed that graphs with chromatic number $r$
contain a {\em subdivision} of $K_r$. 
This conjecture, whose origin is unclear
but attributed to Haj\'os, turned out to be false for $r\ge 7$. Moreover,
it was shown by Erd\H os and Fajtlowicz \cite{ef} that almost all graphs 
are counterexamples.
Albertson conjectured the following.


\begin{conj}
 If $\chi(G)=r$, then $\cro(G)\ge\cro(K_r)$.
\end{conj}

This statement is weaker than Haj\'os' conjecture, since 
if $G$ contains a subdivision of $K_r$ then $\cro(G)\ge\cro(K_r)$.

For $r=5$, Albertson's conjecture
is equivalent to the Four Color Theorem. 
Oporowski and Zhao \cite{oz} verified it for $r=6$, Albertson, Cranston, and Fox \cite{alb}
proved it for $r\le 12$. 
In this note we take one more little step.

\begin{thm}\label{fotetel}
For $r\le 16$, if $\chi(G)=r$, then $\cro(G)\ge\cro(K_r)$.
\end{thm}

In their proof, Albertson, Cranston, and Fox combined
lower bounds for the number of edges of $r$-critical graphs,
and lower bounds on the crossing number of graphs with given number of
vertices and edges. Our proof is very similar, but we use better lower bounds
in both cases.

Albertson, Cranston, and Fox proved that any minimal counterexample 
to Albertson's conjecture should have less than $4r$ vertices.
We slightly improve this result as follows.

\begin{lemma}\label{3.57} 
If $G$ is an $r$-critical graph with $n\ge 3.57r$ vertices, then
$\cro(G)\ge\cro(K_r)$.
\end{lemma}

In Section 2 we review lower bounds for the number of edges of $r$-critical graphs,
in  Section 3 we discuss lower bounds on the crossing number,
and in Section 4 we combine these bounds to obtain the proof of Theorem \ref{fotetel}. 
In Section 5 we prove Lemma \ref{3.57}.

The letter $n$ always denotes the number of vertices of $G$.
In notation and terminology we follow Bondy and Murty \cite{bondy}.
In particular, the {\it join} of two disjoint graphs $G$ and $H$ arises
by adding all edges between vertices of $G$ and $H$.
It is denoted by $G\vee H$.
A vertex $v$ is called {\it simplicial} if it has degree $n-1$.
If a graph $G$ contains a subdivision of $H$, then we also 
say that $G$ contains a {\it topological} $H$.
A vertex $v$ is adjacent to a vertex set $X$ means
that each vertex of $X$ is adjacent to $v$.

\section{Color-critical graphs}

Around 1950, Dirac introduced the concept of color criticality in order to
simplify graph coloring theory,
and it has since led to many beautiful theorems. 
A graph $G$ is $r$-critical if $\chi(G)=r$ but all proper subgraphs of $G$ 
have chromatic number less than $r$.
In what follows, let $G$ denote an $r$-critical graph with $n$ vertices and $m$ edges.

Since $G$ is $r$-critical, every vertex has degree at least $r-1$ and
therefore,\\ $2m\ge (r-1)n$.
Dirac \cite{d57} proved that for $r\ge 3$, if $G$ is not complete, then $2m\ge
(r-1)n+(r-3)$. 
For $r\ge 4$, Dirac \cite{dir} gave a characterization of
$r$-critical graphs with excess $r-3$.
For any fixed $r\ge 3$ let $\Delta_r$ be the family of graphs $G$ 
whose vertex 
set consists of
three non-empty, pairwise disjoint sets $A, B_1, B_2$ with 
$|B_1|+|B_2|=|A|+1=r-1$
and two additional vertices $a$ and $b$ such that $A$ and $B_1\cup B_2$ both
span cliques 
in $G$, they are not connected by any edge, 
$a$ is connected to 
$A\cup B_1$ and $b$ is connected to 
$A\cup B_2$. See Figure \ref{delta}.
Graphs in $\Delta_r$ are called Haj\'os graphs of order 
$2r-1$. 
Observe that  that these graphs have chromatic number $r$ and they
contain a topological $K_r$, hence they 
satisfy Haj\'os' conjecture.

\begin{figure}[ht]
 \begin{center}
  \includegraphics[scale=0.5]{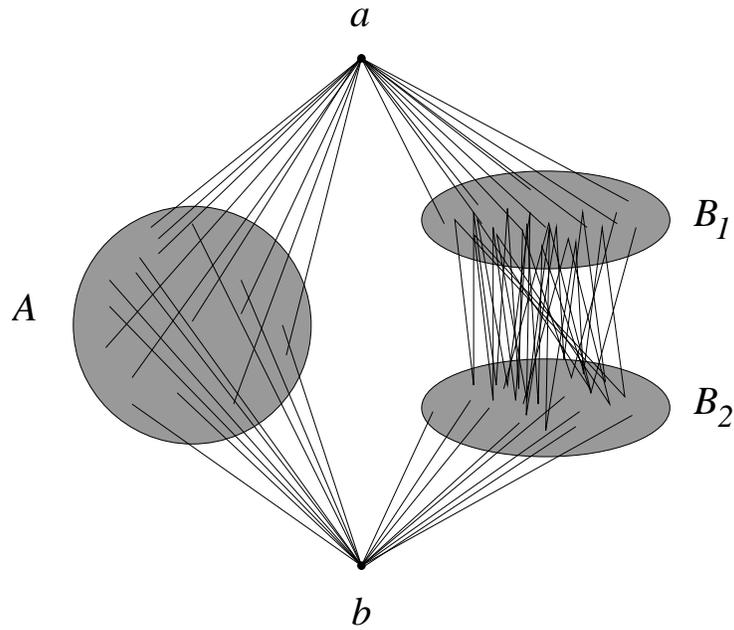}
 \end{center}
  \caption{The family $\Delta_r$}
\label{delta}
\end{figure}

Gallai \cite{gal} proved that $r$-critical graphs with at most $2r-2$ vertices
are the join of two smaller graphs, i.e. their complement is disconnected.
Based on this observation, 
he proved that non-complete  $r$-critical graphs on at most 
$2r-2$ vertices have much larger excess than in Dirac's result. 

\begin{lemma}   
{\rm \cite{gal}} Let $r, p$ be integers satisfying $r\ge 4$ and $2\le p\le r-1$.
If $G$ is an $r$-critical graph with $n=r+p$ vertices, then 
$2m\ge (r-1)n+p(r-p)-2$,
where equality holds if and only if $G$ is the join of $K_{r-p-1}$ and $G\in\Delta_{p+1}$.
\end{lemma}

Since every $G\in\Delta_{p+1}$
contains a topological $K_{p+1}$, the join of $K_{r-p-1}$ and $G$
contains a topological $K_r$. 
This yields a slight improvement for our purposes.

\begin{cor}\label{gallaibound}   
Let $r, p$ be integers satisfying $r\ge 4$ and $2\le p\le r-1$.
If $G$ is an $r$-critical graph with $n=r+p$ vertices, 
and $G$ does not contain a topological $K_r$, then 
$2m\ge (r-1)n+p(r-p)-1$.
\end{cor}

We call the bound given by Corollary \ref{gallaibound} 
the Gallai bound.

For $r\ge 3$, let ${\mathcal E}_r$ denote the family of graphs G,
whose vertex set consists of four non-empty pairwise disjoint sets 
$A_1,A_2,B_1,B_2$, where $|B_1|+|B_2|=|A_1|+|A_2|=r-1$ and 
$|A_2|+|B_2|\le r-1$,
and one additional vertex $c$ such that $A=A_1\cup A_2$ and $B=B_1\cup B_2$
are cliques in $G$, $N_G(c)=A_1\cup B_1$ and a vertex $a\in A$ is adjacent to 
a vertex $b\in B$ if and only if $a\in A_2$ and $b\in B_2$.

\begin{figure}[ht]
 \begin{center}
  \includegraphics[scale=0.5]{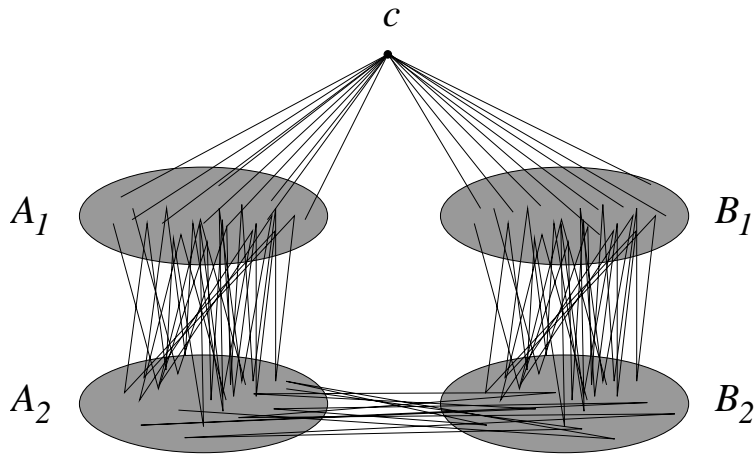}
 \end{center}
  \caption{The family ${\mathcal E}_r$}
\label{er}
\end{figure}

Clearly ${\mathcal E}_r\supset\Delta_r$, and every graph $G\in{\mathcal E}_r$
is $r$-critical with $2r-1$ vertices. 
Kostochka and Stiebitz \cite{kos}
improved the bound of Dirac as follows.

\begin{lemma}
{\rm \cite{kos}} Let $r\ge 4$ and $G$ be an $r$-critical graph. 
If $G$ is neither $K_r$ nor a member of ${\mathcal E}_r$, then 
$2m\ge (r-1)n+(2r-6)$.
\end{lemma}

It is not difficult to prove that any member of ${\mathcal E}_r$ contains a 
topological $K_r$.
Indeed, $A$ and $B$ both span a complete graph on $r-1$ vertices. 
We only have to show that vertex $c$ is connected to $A_2$ or $B_2$ by 
vertex-disjoint paths.
To see this, we observe that $|A_2|$ or $|B_2|$ is the smallest of 
$\{|A_1|,|A_2|,|B_1|,|B_2|\}$.
Indeed, if $|B_1|$ was the smallest, then $|A_2|>|B_1|$ and $|B_2|>|B_1|$ 
implies
$|A_2|+|B_2|>|B_1|+|B_2|=r-1$ contradicting our assumption.
We may assume that $|A_2|$ is the smallest.
Now $c$ is adjacent to $A_1$, and there is a matching of size $|A_2|$ between 
$B_1$ and $B_2$ and between $B_2$ and $A_2$, respectively.
That is, we can find a set $S$ of disjoint paths from $c$ to $A_2$.
In this way $A\cup c\cup S$ is a topological $r$-clique. 

\begin{cor}\label{ksbound}
Let $r\ge 4$ and $G$ be an $r$-critical graph. 
If $G$ does not contain a topological $K_r$ then
$2m\ge (r-1)n+(2r-6)$.
\end{cor}

Let us call this the Kostochka, Stiebitz bound, or KS-bound for short.

In what follows,
we obtain a complete characterization of $r$-critical graphs on $r+3$ or 
$r+4$ vertices.  

\begin{lemma} \label{+3}
For $r\ge 8$, there are precisely two $r$-critical graphs on $r+3$ vertices.
They can be constructed from two $4$-critical graphs on seven vertices by 
adding simplicial vertices.
\end{lemma}

\begin{figure}[ht]
 \begin{center}
  \includegraphics[scale=0.5]{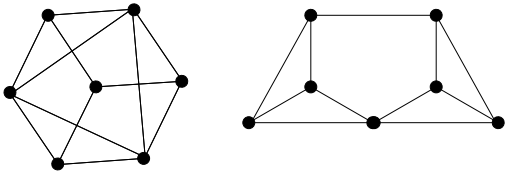}
 \end{center}
  \caption{The two $4$-critical graphs on seven vertices}
\label{4critical}
\end{figure}

\begin{proof}
The proof is by induction on $r$.
For the base case $r=8$, there are precisely two $8$-critical graphs on $11$ 
vertices, see Royle's complete search \cite{roy}.

Let $G$ be an $r$-critical graph with $r\ge 9$ and $n=r+3\ge 12$.
We know that the minimum degree is at least $r-1=n-4$.
If $G$ has a simplicial vertex $v$, then we use induction.
So we may assume that every vertex in $\overline{G}$, the complement of $G$
has degree 1, 2 or 3.
By Gallai's theorem, $\overline{G}$ is disconnected.
Observe the following: if there are at least four independent edges in 
$\overline{G}$, then 
$\chi(G)\le n-4=r-1$, a contradiction.
That is, there are at most three independent edges in $\overline{G}$.
Therefore, $\overline{G}$ has two or three components.
If there is a triangle in the complement, then we can save two colors.
If there were two triangles, then $\chi(G)\le n-4=r-1$, a contradiction.

Assume that there are three components in $\overline{G}$. 
Since each degree is at least one, there are at least three independent edges.
Therefore, there is no triangle in $\overline{G}$ and no path with three edges.
That is, the complement consists of three stars.
Since the degree is at most three and there are at least $12$ vertices,
there is only one possibility: $\overline{G}=K_{1,3}\cup K_{1,3}\cup K_{1,3}$,
see Figure~\ref{harom}.

\begin{figure}[ht]
 \begin{center}
  \includegraphics[scale=0.5]{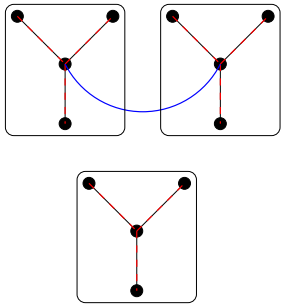}
 \end{center}
  \caption{The complement and a removable edge}
\label{harom}
\end{figure}

We have to check whether this concrete graph is indeed critical.
We observe, that the edge connecting two centers of these stars is not 
critical, a contradiction.

In the remaining case, $\overline{G}$ has two components $H_1$ and $H_2$.
Since there are at most three independent edges, there is one in $H_1$ and
two in $H_2$. It implies that $H_1$ has at most four vertices.
Therefore, $H_2$ has at least eight vertices. 
Consider a spanning tree $T$ of $H_2$ and remove two adjacent vertices of $T$, 
one of them being a leaf. It is easy to see that the remainder of $T$
contains a path with three edges. Therefore, in total we found three 
independent 
edges of $H_2$, a contradiction.
\end{proof}

We need the following result of Gallai.

\begin{thm}\cite{gal} \label{full}
Let $r\ge 3$ and $n< \frac{5}{3}r$.
Then every $r$-critical, $n$-vertex graph contains at least 
$\left\lceil\frac{3}{2}\left(\frac{5}{3}r-n \right)\right\rceil$ 
simplicial vertices.  
\end{thm}

\begin{lemma} \label{+4}
For $r\ge 6$, there are precisely twenty-two $r$-critical graphs on 
$r+4$ vertices.
They can be constructed by adding simplicial vertices to one of the 
following:\linebreak
a $3$-critical graph on seven vertices,\\
four $4$-critical graphs on eight vertices,\\
sixteen $5$-critical graphs on nine vertices, or\\
a $6$-critical graphs on ten vertices.
\end{lemma}

\begin{proof}
For the base of induction, we use Royle's table again, see \cite{roy}.
The full computer search shows that there are precisely twenty-two $6$-critical graphs on ten vertices.
For the induction step, we use Lemma \ref{full} and see that there are at least $r-6$ simplicial vertices.
Since $r\ge 7$, there is always a simplicial vertex. 
We remove it and use the induction hypothesis to finish the proof.
\end{proof}

There is an explicit list of twenty-one $5$-critical graphs on nine vertices \cite{roy}.
We have checked, partly manually, partly using Mader's extremal result 
\cite{mader}, 
that each of those graphs contains a topological $K_5$.
Also the above mentioned $6$-critical graph on ten vertices contains a topological $K_6$.
These results imply the following

\begin{cor}
Any $r$-critical graph on at most $r+4$ vertices satisfy the Haj\'os conjecture. 
\end{cor}

We conjecture that the following slightly more general statement can be proved 
with similar methods.

\begin{conj}
 Let $G$ be an $r$-critical graph on $r+o(r)$ vertices.
Then $G$ satisfies the Haj\'os conjecture.
\end{conj}

\section{The crossing number}

It follows from Euler's formula that a planar graph can have 
at most $3n-6$ edges. Suppose that $G$ has $m\ge 3n-6$ edges.
By deleting crossing edges one by one, it follows by induction that for $n\ge 3$, 

\vskip -3mm

\begin{equation}\label{m}
\cro(G)\ge m-3(n-2)
\end{equation}

Pach et. al. \cite{prtt}  
generalized it and proved the following lower bounds. Each one holds for any graph
$G$ with $n\ge 3$ vertices and $m$ edges.

\vskip -3mm

\begin{equation}\label{7m/3}
\cro(G)\ge 7m/3-25(n-2)/3
\end{equation}

\vskip -3mm

\begin{equation}\label{3m}
\cro(G)\ge 3m-35(n-2)/3
\end{equation}

\vskip -3mm

\begin{equation}\label{4m}
\cro(G)\ge 4m-103(n-2)/6
\end{equation}

\vskip -3mm

\begin{equation}\label{5m}
\cro(G)\ge 5m-25(n-2)
\end{equation}

Inequality (\ref{m}) is the best for $m\le 4(n-1)$, 
           (\ref{7m/3}) is the best for $4(n-2)\le m\le 5(n-2)$,
           (\ref{3m}) is the best for $5(n-2)\le m\le 5.5(n-2)$,
           (\ref{4m}) is the best for $5.5(n-2)\le m\le 47(n-2)/6$, and
           (\ref{5m}) is the best for $47(n-2)/6\le m$.

It was also shown in \cite{prtt} that  (\ref{m}) can not be improved in the
range $m\le 4(n-1)$,\linebreak
and (\ref{7m/3}) can not be improved in the
range $4(n-2)\le m\le 5(n-2)$, apart from an additive constant.
The other inequalities are conjectured to be 
far from optimal. 
Using the methods in  \cite{prtt} one can obtain an infinite family of 
such linear inequalities,
of the form $am-b(n-2)$.

The most important inequality for crossing numbers is undoubtedly 
the {\em Crossing Lemma}, first proved by Ajtai, Chv\'atal, Newborn,
Szemer\'edi \cite{acns}, and independently by Leighton \cite{l}.
If $G$ has $n$ vertices and $m\ge 4n$ edges, then 
\begin{equation}\label{1/64}
\cro(G)\ge \frac{1}{64}\frac{m^3}{n^2}.
\end{equation}
The original constant was much larger, the constant 
$\frac{1}{64}$ comes from the well-known probabilistic proof of 
Chazelle, Sharir, and Welzl \cite{az}. The basic idea is to take 
a random spanned subgraph and apply inequality (\ref{m}) for that.

The order of magnitude of this bound can not be improved, see \cite{prtt},
the best known constant is obtained in \cite{prtt}.
If $G$ has $n$ vertices and $m\ge \frac{103}{16}n$ edges, then 
\begin{equation}\label{1/31.1}
\cro(G)\ge \frac{1}{31.1}\frac{m^3}{n^2}.
\end{equation}
The proof is very similar to the proof of (\ref{1/64}), the 
main difference is that 
instead of (\ref{m}), inequality (\ref{4m}) is applied for the random subgraph.
The proof of the following technical lemma is based on the same idea.

\begin{lemma}\label{ronda}
Suppose that $n\ge 10$, 
and $0< p\le 1$.
Let 
$$\cro(n, m, p)= \frac{4m}{p^2}-\frac{103n}{6p^3}+\frac{103}{3p^4}-\frac{5n^2(1-p)^{n-2}}{p^4}.$$
Then for any graph $G$ with $n$ vertices and $m$ edges
$$\cro(G)\ge\cro(n, m, p).$$
\end{lemma}

\begin{proof}
Observe that inequality (\ref{4m})
does not hold for graphs with
at most two vertices. 
For any graph $G$, let
\[ \cro'(G) = \left\{ \begin{array}{ll}
\cro(G) & {\mbox{if }} n\ge 3\\
4       & {\mbox{if }} n=2\\
18      & {\mbox{if }} n=1\\
35      & {\mbox{if }} n=0
\end{array}\right.
\]

It is easy to see that for {\em any} graph $G$
\begin{equation}\label{103'}
\cro'(G)\ge 4m-\frac{103}{6}(n-2).
\end{equation}

Let $G$ be a graph with $n$ vertices and $m$ edges. 
Consider a drawing of $G$  with $\cro(G)$ crossings.
Choose each vertex of $G$ independently with probability $p$, and let $G'$ be 
a subgraph of $G$ spanned by the selected vertices. 
Consider the drawing of $G'$ 
{\em inherited} from the drawing of $G$, that is, each edge of $G'$ is 
drawn exactly as it is drawn in $G$. 
Let $n'$ and $m'$ be the number of vertices and edges of $G'$,
and let $x$ be the number of crossings in the present drawing of $G'$.
Using that $E(n')=pn$, $E(m')=p^2m$, $E(x)=p^4\cro(G)$, and the linearity of
expectations,
$$E(x)\ge E(\cro(G'))\ge E(\cro'(G'))-4P(n'=2)-18P(n'=1)-35P(n'=0)\ge$$
$$\ge 4p^2m-\frac{103}{6}pn+\frac{103}{3}-4{n\choose 2}p^2(1-p)^{n-2}
-18np(1-p)^{n-1}-35(1-p)^n\ge$$
$$\ge 4p^2m-\frac{103}{6}pn+\frac{103}{3}-5n^2(1-p)^{n-2}.$$

Dividing by $p^4$ we obtain the statement of the Lemma.
\end{proof}

Note that in our applications $p$ will be at least $1/2$, $n$ will be at least
13,
therefore, the last term in the inequality, $\frac{5n^2(1-p)^{n-2}}{p^4}$,
will be negligible.


We also need some bounds on the crossing number of the complete graph,
$\cro(K_r)$.
It is not hard to see that
\begin{equation}\label{Z(r)}
\cro(K_r)\le Z(r)=
\frac{1}{4}\left\lfloor\frac{r}{2}\right\rfloor\left\lfloor\frac{r-1}{2}\right\rfloor 
\left\lfloor\frac{r-2}{2}\right\rfloor\left\lfloor\frac{r-3}{2}\right\rfloor,
\end{equation}
see e. g. \cite{rt}.  
Guy conjectured \cite{g} that $\cro(K_r)=Z(r)$. This conjecture has been
verified for $r\le 12$ but still open for $r>12$.
The best known lower bound is due to de Klerk et. al. \cite{kmprs}:
$\cro(K_r)\ge 0.86Z(r)$.


\section{Proof of Theorem \ref{fotetel}}

Suppose that $G$ is an $r$-critical graph.
If $G$ contains a topological $K_r$, then 
clearly $\cro(G)\ge\cro(K_r)$. Suppose in the sequel that 
$G$ does not contain a topological $K_r$.

Therefore, we can apply the Kostochka, Stiebitz, 
and the Gallai bounds on the number of edges.
Then we use Lemma \ref{ronda} to get the desired lower bound on the
crossing number. Albertson et. al. \cite{alb} used the same approach, but they
used a weaker version of the Kostochka, Stiebitz, 
and the Gallai bounds, and instead of Lemma \ref{ronda} they applied the 
weaker inequality (\ref{4m}). 
In the next table, we include the results of our calculations.
For comparison, we also included the result Albertson et al. might have had using (\ref{4m}).
In the Appendix we present our simple Maple program performing all
calculations.

\bigskip

1. Let $r=13$. By (\ref{Z(r)}) we have $\cro(K_{13})\le 225$.

\bigskip

{\small
\begin{tabular}{|c|c|c|c|c|}
\hline
$n$ & $e$ & bound (\ref{4m}) & $p$ & $\lceil\cro(n,m,p)\rceil$ \\
\hline
18 & 128 & 238 & 0.719 & 288 \\
19 & 135 & 249 & 0.732 & 296 \\
20 & 141 & 255 & 0.751 & 298 \\
21 & 146 & 258 & 0.774 & 294 \\
\hline
\end{tabular}
}
\medskip

If $n\ge 22$, then the KS-bound combined with (\ref{4m}) gives the desired result.\\
$2m\ge 12 n+20\Rightarrow \cro(G)\ge 4(6 n+10)-103/6 (n-2)\ge 224.67$,
if $n\ge 22$.

\bigskip

2. Let $r=14$. By (\ref{Z(r)}) we have $\cro(K_{14})\le 315$.

\bigskip

{\small
\begin{tabular}{|c|c|c|c|c|}
\hline
$n$ & $e$ & bound (\ref{4m}) & $p$ & $\lceil\cro(n,m,p)\rceil$ \\
\hline
19 & 146 & 293 & 0.659 & 388 \\
20 & 154 & 307 & 0.670 & 402 \\
21 & 161 & 318 & 0.684 & 407 \\
22 & 167 & 325 & 0.702 & 406 \\
23 & 172 & 328 & 0.723 & 398 \\
24 & 176 & 327 & 0.747 & 384 \\
25 & 179 & 322 & 0.775 & 366 \\
26 & 181 & 312 & 0.807 & 344 \\
\hline
\end{tabular}
}
\medskip

If $n\ge 27$, then the KS-bound combined with (\ref{4m}) gives the desired result.\\
$2m\ge 13 n+22\Rightarrow \cro(G)\ge 4(6.5 n+11)-103/6 (n-2)\ge 316$,
if $n\ge 27$.

\bigskip

3. Let $r=15$. By (\ref{Z(r)}) we have $\cro(K_{15})\le 441$.

\bigskip

{\small
\begin{tabular}{|c|c|c|c|c|}
\hline
$n$ & $e$ & bound (\ref{4m})   & $p$ & $\lceil\cro(n,m,p)\rceil$ \\
\hline
20 & 165 & 351 & 0.610 & 510 \\
21 & 174 & 370 & 0.617 & 531 \\
22 & 182 & 385 & 0.623 & 542 \\
23 & 189 & 396 & 0.642 & 545 \\
24 & 195 & 403 & 0.659 & 539 \\
25 & 200 & 406 & 0.678 & 526 \\
26 & 204 & 404 & 0.700 & 508 \\
27 & 207 & 399 & 0.725 & 484 \\
\hline
\end{tabular}
}

\medskip

%
%
Suppose now that $G$ is $15$-critical and 
$n\ge 28$. By the KS-bound we have $m\ge 7n+12$.
%
%
Apply Lemma \ref{ronda} with $p=0.764$ and a straightforward calculation gives 
$\cro(G)\ge \cro(n,m,0.764)\ge 441$.
%

\medskip

4. Let $r=16$. By (\ref{Z(r)}) we have $\cro(K_{16})\le 588$.

\bigskip

{\small
\begin{tabular}{|c|c|c|c|c|}
\hline
$n$ & $e$ & bound (\ref{5m})  & $p$ & $\lceil\cro(n,m,p)\rceil$ \\
\hline
21 & 185 & 450 & 0.567 & 657 \\
22 & 195 & 475 & 0.573 & 687 \\
23 & 204 & 495 & 0.581 & 706 \\
24 & 212 & 510 & 0.592 & 714 \\
25 & 219 & 520 & 0.605 & 712 \\
26 & 225 & 525 & 0.621 & 701 \\
27 & 230 & 525 & 0.639 & 683 \\
28 & 234 & 520 & 0.659 & 658 \\
29 & 237 & 510 & 0.681 & 628 \\
30 & 239 & 495 & 0.706 & 593 \\
31 & 246 & 505 & 0.713 & 601 \\
\hline
\end{tabular}
}
\medskip

Suppose now that $G$ is $16$-critical and 
$n\ge 32$. By the KS-bound we have $m\ge 7.5n+13$.
Apply Lemma \ref{ronda} with $p=0.72$ and again a straightforward calculation gives 
$\cro(G)\ge \cro(n,m,0.72)\ge 588$.

%


This concludes the proof of Theorem \ref{fotetel}.

\bigskip

\noindent {\bf Remark.}

For $r\ge 17$ we could not completely verify Albertson's conjecture.
The next table contains our calculations for $r=17$. 
There are three cases, $n=32,33,34$, for which our approach is not sufficient.  
By (\ref{Z(r)}) we have $\cro(K_{17})\le 784$.

\bigskip

{\small
\begin{tabular}{|c|c|c|c|c|}
\hline
$n$ & $e$ & bound from         & $p$ & bound using \\
    &     & equation \ref{5m} &   & $\cro(n, e, p)$ \\
\hline
22 & 206 & 530 & 0.530 & 832 \\
23 & 217 & 560 & 0.534 & 874 \\
24 & 227 & 585 & 0.541 & 902 \\
25 & 236 & 605 & 0.550 & 917 \\
26 & 244 & 620 & 0.560 & 920 \\
27 & 251 & 630 & 0.573 & 913 \\
28 & 257 & 635 & 0.588 & 897 \\
29 & 262 & 635 & 0.604 & 872 \\
30 & 266 & 630 & 0.622 & 840 \\
31 & 269 & 620 & 0.643 & 802 \\
\hline
32 & 271 & 605 & 0.665 & 759 \\
33 & 278 & 615 & 0.672 & 765 \\
34 & 286 & 630 & 0.677 & 779 \\
\hline
\end{tabular}
}
\medskip

\begin{lemma}
Let $G$ be a $17$-critical graph on $n$ vertices.
If $n\ge 35$, then $\cro(G)\ge 784\ge \cro(K_{17})$.
\end{lemma}

\begin{proof}
Let $p=0.681$. 
Then 
$\cro(G)\ge \cro(n,m,0.681)\ge 14.64n+280.38$.
Therefore, if $n\ge \frac{784-280.38}{14.64}\ge 34.4$, then we are done.
(Without the probabilistic argument, the same result holds with $n\ge 44$.)
\end{proof}

\begin{lemma}
 Let $G$ be a $17$-critical graph on $32$ vertices.
Then $\cro(G)\ge \cro(K_{17})$.
\end{lemma}

\begin{proof}
Gallai \cite{gal} proved that any $r$-critical graph on at most $2r-2$ vertices is a
join of two smaller critical graphs. This is a structural version of the Gallai bound.
In our case, $r=17$, and $n=2r-2=32$. 
Assume that $G=G_1\vee G_2$, where $G_1$ is $r_1$-critical on $n_1$ vertices, $G_2$ is $r_2$-critical
on $n_2$ vertices, where $17=r_1+r_2$ and $32=n_1+n_2$.
The sum of the degrees of $G$ can be estimated as the sum of the degrees of the vertices in $G_i$,
for $i=1,2$, plus twice the number of edges between $G_1$ and $G_2$:
$2m\ge (r_1-1)n_1+(r_2-1)n_2+2(r-3)+2n_1n_2$.

How much do we gain with this calculation compared to the direct application of the Gallai bound on $G$?
That is seen after a simple subtraction:\\
$(r_1-1)n_1+(r_2-1)n_2+2(r-3)+2n_1n_2-(r-1)n-2(r-3)=(n_1-r_1)n_2+(n_2-r_2)n_1$.\\
This value is minimal if $n_2=r_2=1$. In that case, we gain $n_1-r_1=15$.
That is, in our calculation we can add $\lceil 15/2\rceil$ edges, after which 
$\cro (G)\ge 834$ arises.
\end{proof}

It is clear that our improvement on Gallai's result relies on the fact that Kostochka and Stiebitz improved
Dirac's result.

\section{Proof of Lemma \ref{3.57}}

Suppose that $r\ge 17$ and $G$ is an $r$-critical graph with $n$ vertices and $m$ edges.
If $n\ge 4r$ then the statement holds by \cite{alb}.
Suppose that 
$3.57r\le n\le 4r$.
In order to estimate the crossing number of $G$,
instead of the probabilistic argument in the proof of Lemma \ref{ronda},
we apply inequality (\ref{4m}) for each spanned subgraph of $G$ 
with exactly 52 vertices.
Let $k={n\choose 52}$ and let $G_1, G_2, \ldots , G_k$ 
be the spanned subgraphs of $G$ with 52 vertices.
Suppose that $G_i$ has $m_i$ edges.
Then for any $i$, by (\ref{4m}) we have 
$$\cro(G_i)\ge 4m_i-\frac{103}{6}\cdot 50,$$
consequently, 
$$\cro(G)\ge {1\over {n-4\choose 48}}\sum_{i=1}^{k}
\left(4m_i-{103\over 6}\cdot 50\right)
={4m\over {n-4\choose 48}}{n-2\choose 50}-
{50\over {n-4\choose 48}}{103\over
6}{n\choose 52}=$$
$$={4(n-2)(n-3)m\over 50\cdot 49}
-{103\over 6}{n(n-1)(n-2)(n-3)\over 52\cdot 51\cdot 49}=$$
$$\ge {2(n-2)(n-3)n(r-1)\over 50\cdot 49}
-{103\over 6}{n(n-1)(n-2)(n-3)\over 52\cdot 51\cdot 49}=$$
$$={n(n-2)(n-3)\over 49}\left({r-1\over 25}-{103(n-1)\over 6\cdot 52\cdot 51}\right)$$
since we counted each possible crossing at most ${n-4\choose 48}$
times, and each edge of $G$ exactly  ${n-2\choose 50}$ times.

Finally, some calculation shows that it is greater than
$${1\over 64}r(r-1)(r-2)(r-3)>\cro(K_r)$$
which proves the lemma. \hfill $\Box$


\section*{Remarks}

1. As we have already mentioned, see (\ref{1/31.1}), the best known constant in the Crossing Lemma  $1/31.1$
is obtained in \cite{prtt}.
Montaron \cite{montaron} managed to improve it slightly for {\em dense} graphs, that
is,
in the case when $m=O(n^2)$. His calculations are similar to the proof of
Lemmas \ref{3.57} and \ref{ronda}.

\smallskip

2. 
Our attack of the Albertson conjecture is based on the following philosophy.
We calculate a lower bound for the number of edges of an $r$-critical $n$-vertex graph $G$.
Then we substitute this into the lower bound given by Lemma \ref{ronda}.
Finally, we compare the result and the Zarankiewicz number $Z(r)$. 
For large $r$, this method is not sufficient, but it gives the right order of
magnitude,
and the constants are roughly within a factor of $4$.

Let $G$ be an $r$-critical graph with $n$ vertices, where $r\le n\le 3.57r$.
Then $2m\ge (r-1)n$.
We can apply (\ref{1/31.1}):

$$\cro(G)\ge \frac{1}{31.1}\frac{((r-1)n/2)^3}{n^2}=\frac{(r-1)^3n}{31.1\cdot
  8}\ge \frac{1}{250}r(r-1)^3\ge \frac{Z(r)}{4}.$$

\smallskip

3. Let $G=G(n, p)$ be a random graph with $n$ vertices and edge probability
$p=p(n)$.
It is known (see \cite{jlr}) that there is a constant $C_0>0$ such that
if $np>C_0$ then asymptotically almost surely we have
$$\chi(G)<\frac{np}{\log{np}}.$$
Therefore, asymptotically almost surely
$$\cro(K_{\chi(G)})\le Z(\chi(G))< \frac{n^4p^4}{64\log^4{np}}.$$
On the other hand, by \cite{pt}, if $np>20$ then almost 
surely 
$$\cro(G)\ge \frac{n^4p^2}{20000}.$$
Consequently, almost surely we have $\cro(G)>\cro(K_{\chi(G)})$, that is, 
roughly speaking, unlike in the case of the Haj\'os conjecture, 
a random graph almost surely satisfies the statement of the Albertson 
conjecture.

\smallskip

4. If we do not believe in Albertson's conjecture, 
we have to look for a counterexample in the range $n\le 3.57r$. 
Any candidate must also be a counterexample for the Haj\'os Conjecture. 
It is tempting to look at Catlin's graphs.

Let $C_5^k$ denote the graph arising from $C_5$ by repeating each vertex $k$ times.
That is, each vertex of $C_5$ is blown up to a complete graph on $k$ vertices and
any edge of $C_5$ is blown up to a complete bipartite graph $K_{k,k}$.

  \begin{lemma}
Catlin's graphs satisfy the Albertson conjecture.
  \end{lemma}

\begin{proof}
It is known that $\chi(C_5^k)=\lceil\frac{5}{2}k\rceil$.
To draw $C_5^k$, there must be two copies of $K_{2k}$, a $K_k$ and three copies of $K_{k,k}$
drawn. Therefore

$$cr(C_5^k)\ge 2 Z(2k)+Z(k)+3cr(K_{k,k})\sim 2\frac{1}{4}k^4+\frac{1}{4}\left(\frac{k}{2}\right)^4+3\left(\frac{k}{2}\right)^4>0.70 k^4.$$

On the other hand

\begin{equation}
 cr(K_{\chi(C_5^k)})\sim cr(K_{\frac{5}{2}k})\le \frac{1}{4} \left(\frac{5}{4}k\right)^4<0.62 k^4
\end{equation}

which proves the claim.
\end{proof}


\section*{Appendix}

{\tt
\noindent start:=proc(r,n)\\
  local p,m,eredm,f,g,h,cr;\\
  if (n$<$=2*r-2) then\\
           p:=n-r;\\
           m:=ceil(((r-1)*n+p*(r-p)-1)/2);\\
  else\\
           m:=ceil(((r-1)*n+2*(r-3))/2);\\
  fi;\\
 g:= ceil(5*m-25*(n-2));\\
 print(m,g);\\
 f:= 4*m*x\^{}2-(103/6)*n*x\^{}3+(103/3)*x\^{}4;\\
 eredm:=[solve((diff(f,x)/x)=0, x)];\\
 print(evalf(eredm));\\
 cr := min(eredm[1], eredm[2]);\\
 print(evalf(1/cr));\\
 h:= f-(5*n\^{}2*(1-1/x)\^{}(n-2))/(1/x)\^{}4;\\
 evalf((subs(x=cr, h)));\\
 end:\\
}

\end{document}